\documentclass[12pt]{article}
\usepackage{amsfonts}
\usepackage{amsmath}
\usepackage{fullpage}
\usepackage{graphicx}
\usepackage{url}

\usepackage{amsthm}
\newtheorem{lemma}{Lemma}
\newtheorem{theorem}{Theorem}

\newcommand{\RR}{{\mathbb R}}
\newcommand{\CC}{{\mathbb C}}
\newcommand{\HH}{{\mathbb H}}
\newcommand{\KK}{{\mathbb K}}
\newcommand{\OO}{{\mathbb O}}

\renewcommand{\AA}{\textbf{A}}
\newcommand{\BB}{\textbf{B}}

\newcommand{\II}{\textbf{I}}
\newcommand{\XX}{\textbf{X}}
\newcommand{\YY}{\textbf{Y}}
\newcommand{\ZZ}{\textbf{0}}

\newcommand{\eps}{\hbox{\boldmath$\epsilon$}}
\newcommand{\bphi}{\hbox{\boldmath$\phi$}}
\newcommand{\Xt}{\widetilde{\XX}}

\newcommand{\AAA}{{\cal A}}
\newcommand{\BBB}{{\cal B}}
\newcommand{\QQQ}{{\cal Q}}
\newcommand{\III}{{\cal I}}
\newcommand{\XXX}{{\cal X}}
\newcommand{\YYY}{{\cal Y}}

\newcommand{\HHH}{\mathrm{H}_3(\OO)}
\newcommand{\Pc}{\hbox{\boldmath${\cal P}$}}

\newcommand{\isom}{\cong}

\newcommand{\tr}{\textrm{tr}}
\renewcommand{\Im}{\textrm{Im}}

\newcommand{\so}{\mathfrak{so}}
\newcommand{\su}{\mathfrak{su}}
\renewcommand{\sl}{\mathfrak{sl}}	
\renewcommand{\sp}{\mathfrak{sp}}	
\renewcommand{\aa}{\mathfrak{a}}

\newcommand{\cc}{\mathfrak{c}}
\newcommand{\dd}{\mathfrak{d}}
\newcommand{\ee}{\mathfrak{e}}
\newcommand{\ff}{\mathfrak{f}}
\renewcommand{\gg}{\mathfrak{g}}	

\newcommand{\SO}{\hbox{SO}}
\newcommand{\SU}{\hbox{SU}}

\newcommand{\Sp}{\hbox{Sp}}
\newcommand{\EE}{E}

\begin{document}


\title{\vspace{-0.5in}
{\bfseries\boldmath A Symplectic Representation of $\EE_7$}
}

\author{
	Tevian Dray \\[-2.5pt]
	\normalsize
	\textit{Department of Mathematics, Oregon State University,
		Corvallis, OR  97331, USA} \\[-2.5pt]
	\normalsize
	\texttt{tevian@math.oregonstate.edu} \\
	\and
	Corinne A. Manogue \\[-2.5pt]
	\normalsize
	\textit{Department of Physics, Oregon State University,
		Corvallis, OR  97331, USA} \\[-2.5pt]
	\normalsize
	\texttt{corinne@physics.oregonstate.edu} \\
	\and
	Robert A. Wilson \\[-2.5pt]
	\normalsize
	\textit{School of Mathematical Sciences, Queen Mary,
		University of London, London E1 4NS, UK} \\[-2.5pt]
	\normalsize
	\texttt{R.A.Wilson@qmul.ac.uk} \\
}

\date{\normalsize October 26, 2013}

\maketitle

\vspace{-15pt}
\begin{abstract}
We explicitly construct a particular real form of the Lie algebra $\ee_7$ in
terms of symplectic matrices over the octonions, thus justifying the
identifications $\ee_7\isom\sp(6,\OO)$ and, at the group level,
$\EE_7\isom\Sp(6,\OO)$.  Along the way, we provide a geometric description of
the minimal representation of $\ee_7$ in terms of rank 3 objects called
\textit{cubies}.
\end{abstract}

\section{Introduction}
\label{intro}

\begin{table}
\small
\begin{center}
\begin{tabular}{|c|c|c|c|c|}
\hline
&$\RR$&$\CC$&$\HH$&$\OO$\\\hline
$\RR'$
&$\su(3,\RR)$&$\su(3,\CC)$&$\cc_3\isom\su(3,\HH)$&$\ff_4\isom\su(3,\OO)$\\
\hline
$\CC'$&
$\sl(3,\RR)$&$\sl(3,\CC)$&$\aa_{5(-7)}\isom\sl(3,\HH)$
  &$\ee_{6(-26)}\isom\sl(3,\OO)$\\
\hline
$\HH'$&
$\cc_{3(3)}\isom\sp(6,\RR)$&$\su(3,3,\CC)$&$\dd_{6(-6)}$&$\ee_{7(-25)}$\\
\hline
$\OO'$&
$\ff_{4(4)}$&$\ee_{6(2)}$&$\ee_{7(-5)}$&$\ee_{8(-24)}$\\
\hline
\end{tabular}
\end{center}
\caption{The ``half-split'' $3\times3$ magic square of Lie algebras.}
\label{3x3}
\end{table}

The Freudenthal-Tits magic square~\cite{Freudenthal,Tits} of Lie algebras
provides a parametrization in terms of division algebras of a family of Lie
algebras that includes all of the exceptional Lie algebras except $\gg_2$.
The ``half-split'' version of the magic square, in which one of the division
algebras is split, is given in Table~\ref{3x3}.  The interpretation of the Lie
algebra real forms appearing in the first two rows of the magic square as
$\su(3,\KK)$ and $\sl(3,\KK)$ has been discussed in~\cite{Denver,York}; see
also~\cite{Sudbery,SudberyBarton}.  Freudenthal~\cite{FreudenthalE7} provided
an algebraic description of the symplectic geometry of $\ee_7$, and Barton \&
Sudbery~\cite{SudberyBarton} advanced this description to the Lie algebra
level by interpreting the third row of the magic square as $\sp(6,\KK)$.  We
continue this process here, by providing a natural symplectic interpretation
of the minimal representation of $\ee_7=\ee_{7(-25)}$.

\section{\boldmath Freudenthal's Description of $\ee_7$}
\label{e7Freud}

Let $\XXX,\YYY\in\HHH$ be elements of the Albert algebra, that is,
$3\times3$ Hermitian matrices whose components are octonions.  There are two
natural products on the Albert algebra, namely the \textit{Jordan product}
\begin{equation}
\XXX\circ\YYY = \frac12 \Bigl( \XXX\YYY + \YYY\XXX \Bigr)
\end{equation}
and the \textit{Freudenthal product}
\begin{equation}
\XXX*\YYY
  = \XXX\circ\YYY
	- \frac12\Bigl( (\tr\XXX)\YYY + (\tr\YYY)\XXX \Bigr)
	+ \frac12\Bigl( (\tr\XXX)(\tr\YYY)-\tr(\XXX\circ\YYY) \Bigr) \III
\end{equation}
which can be thought of as a generalization of the cross product on $\RR^3$
(with the trace of the Jordan product playing the role of the dot product).

The Lie algebra $\ee_6=\ee_{6(-26)}$ acts on the Albert algebra $\HHH$.
The generators of $\ee_6$ fall into one of three categories; there are 26
\textit{boosts}, 14 \textit{derivations} (elements of $\gg_2$), and 38
remaining \textit{rotations} (the remaining generators of $\ff_4$).  For both
boosts and rotations, $\phi\in\ee_6$ can be treated as a $3\times3$,
tracefree, octonionic matrix; boosts are Hermitian, and rotations are
anti-Hermitian.  Such matrices $\phi\in\ee_6$ act on the Albert algebra via
\begin{equation}
\XXX \longmapsto \phi\XXX + \XXX\phi^\dagger
\label{e6act}
\end{equation}
where $\dagger$ denotes conjugate transpose (in $\OO$).  Since the derivations
can be obtained by successive rotations (or boosts) through \textit{nesting},
it suffices to consider the boosts and rotations, that is, to consider matrix
transformations.
\footnote{Since all rotations can be obtained from pairs of boosts, it would
be enough to consider boosts alone.}

The dual representation of $\ee_6$ is formed by the duals $\phi'$ of each
$\phi\in\ee_6$, defined via
\begin{equation}
\tr\bigl(\phi(\XXX)\circ \YYY\bigr) = -\tr\bigl(\XXX\circ \phi'(\YYY)\bigr)
\end{equation}
for $\XXX,\YYY\in\HHH$.  It is easily checked that $\phi'=\phi$ on
rotations, but that $\phi'=-\phi$ on boosts.  Thus,
\begin{equation}
\phi' = -\phi^\dagger
\label{phiadj}
\end{equation}
for both boosts and rotations.

We can regard $\ee_7$ as the conformal algebra associated with $\ee_6$, since
$\ee_7$ consists of the 78 elements of~$\ee_6$, together with 27 translations,
27 conformal translations, and a dilation.  In fact,
Freudenthal~\cite{FreudenthalE7} represents elements of $\ee_7$ as
\begin{equation}
\Theta = (\phi,\rho,\AAA,\BBB)
\label{ThetaDef}
\end{equation}
where $\phi\in\ee_6$, $\rho\in\RR$ is the dilation, and $\AAA,\BBB\in\HHH$ are
elements of the Albert algebra, representing (null) translations.

What does $\Theta$ act on?  Freudenthal~\cite{FreudenthalE7} explicitly
constructs the minimal representation of~$\ee_7$, which consists of elements
of the form
\begin{equation}
\Pc=(\XXX,\YYY,p,q)
\label{Pcdef}
\end{equation}
where $\XXX,\YYY\in\HHH$, and $p,q\in\RR$.
But how are we to visualize these elements?
Freudenthal does tell us that $\Theta$ acts on $\Pc$ via
\begin{align}
  \XXX &\longmapsto \phi(\XXX) + \frac13\,\rho\,\XXX + 2 \BBB*\YYY + \AAA\,q
\label{FreudX}\\
  \YYY &\longmapsto 2 \AAA*\XXX + \phi'(\YYY) - \frac13\,\rho\,\YYY + \BBB\,p
\label{FreudY}\\
  p &\longmapsto \tr(\AAA\circ \YYY) - \rho\,p
\label{pFreud}\\
  q &\longmapsto \tr(\BBB\circ \XXX) + \rho\,q 
\label{qFreud}
\end{align}
But again, how are we to visualize this action?

We conclude this section by giving two further constructions due to
Freudenthal~\cite{FreudenthalE7}.  There is a ``super-Freudenthal'' product~$*$
taking elements $\Pc$ of the minimal representation of $\ee_7$ to elements of
$\ee_7$, given by
\footnote{We use $*$ to denote this ``super-Freudenthal'' product because of
its analogy to the Freudenthal product~$*$, with which there should be no
confusion.  Neither of these products is the same as the Hodge dual map, also
denoted $*$, used briefly in Sections~\ref{sp4} and~\ref{cubies}.}
\begin{equation}
\Pc*\Pc = (\phi,\rho,\AAA,\BBB)
\label{PstarP0}
\end{equation}
where
\begin{align}
  \phi &= \langle \XXX,\YYY\rangle \label{PstarP1}\\
  \rho &= -\frac14 \tr\bigl(\XXX\circ\YYY - pq\,\III\bigr) \\
  \AAA &= -\frac12 \bigl(\YYY*\YYY - p\,\XXX\bigr) \\
  \BBB &= \frac12 \bigl(\XXX*\XXX - q\,\YYY\bigr) \label{PstarP4}
\end{align}
where
\begin{equation}
\langle X,Y\rangle Z = Y\circ(X\circ Z) - X\circ(Y\circ Z)
			- (X\circ Y)\circ Z + \frac13 \tr(X\circ Y) Z
\end{equation}
Finally, $\ee_7$ preserves the quartic invariant
\begin{equation}
J = \tr\bigl((\XXX*\XXX)\circ(\YYY*\YYY)\bigr) - p\det\XXX - q\det\YYY
	-\frac14 \bigl( \tr(\XXX\circ\YY)-pq \bigr)^2
\label{quartic}
\end{equation}
which can be constructed using $\Pc*\Pc$.

\section{\boldmath The Symplectic Structure of $\so(k+2,2)$}
\label{sp4}

\begin{table}
\small
\begin{center}
\begin{tabular}{|c|c|c|c|c|}
\hline
&$\RR$&$\CC$&$\HH$&$\OO$\\\hline
$\RR'$&
 $\so(2)\isom\su(2,\RR)$&$\so(3)\isom\su(2,\CC)$&
 $\so(5)\isom\su(2,\HH)$&$\so(9)\isom\su(2,\OO)$\\
\hline
$\CC'$&
 $\so(2,1)\isom\sl(2,\RR)$&$\so(3,1)\isom\sl(2,\CC)$&
 $\so(5,1)\isom\sl(2,\HH)$&$\so(9,1)\isom\sl(2,\OO)$\\
\hline
$\HH'$&
 $\so(3,2)\isom\sp(4,\RR)$&$\so(4,2)\isom\su(2,2,\CC)$&
 $\so(6,2)$&$\so(10,2)$\\
\hline
$\OO'$&
 $\so(5,4)$&$\so(6,4)$&$\so(8,4)$&$\so(12,4)$\\
\hline
\end{tabular}
\end{center}
\caption{The ``half-split'' $2\times2$ magic square of Lie algebras.}
\label{2x2}
\end{table}

An analogous problem has been analyzed for the $2\times2$ magic square, which
is shown in Table~\ref{2x2}; the interpretation of the first two rows was
discussed in~\cite{Lorentz}; see also~\cite{Sudbery}.  Dray, Huerta, and Kincaid
showed first~\cite{SO42} (see also~\cite{JoshuaThesis}) how to relate
$\SO(4,2)$ to $\SU(2,\HH'\otimes\CC)$, and later~\cite{2x2} extended
their treatment to the full $2\times2$ magic square of Lie groups in
Table~\ref{2x2}.  In the third row, their Clifford algebra description of
$\SU(2,\HH'\otimes\KK)$ is equivalent to a symplectic description as
$\Sp(4,\KK)$, with $\KK=\RR,\CC,\HH,\OO$.

Explicitly, they represent $\so(k+2,2)$, where $k=|\KK|=1,2,4,8$, in terms of
actions on $4\times4$ matrices of the form
\begin{equation}
P_0 =
\begin{pmatrix} p\,\II & \XX \\ -\Xt & q\,\II \end{pmatrix}
\label{Pdef}
\end{equation}
where $\XX$ is a $2\times2$ Hermitian matrix over $\KK$, representing
$\so(k+1,1)$, $p,q\in\RR$, $\II$ denotes the $2\times2$ identity matrix, and
tilde denotes trace-reversal, that is, $\Xt=\XX-\tr(\XX)\,\II$.  The matrix
$P_0$ can be thought of as the upper right $4\times4$ block of an $8\times8$
Clifford algebra representation, and the action of $\so(k+2,2)$ on $P_0$ is
obtained as usual from (the restriction of) the quadratic elements of the
Clifford algebra.  The generators $A\in\so(k+2,2)$ can be chosen so that the
action takes the form
\begin{equation}
P_0 \longmapsto A P_0 \pm P_0 A
\label{soact}
\end{equation}
where the case-dependent signs are related to the restriction from $8\times8$
matrices to $4\times4$ matrices.  Following Sudbery~\cite{Sudbery}, we define
the elements $A$ of the symplectic Lie algebra $\sp(4,\KK)$ by the condition
\begin{equation}
A \Omega + \Omega A^{\dagger} = 0
\label{sp4def}
\end{equation}
where
\begin{equation}
\Omega = \begin{pmatrix} \ZZ& \II\\ -\II& \ZZ \end{pmatrix}
\end{equation}
Solutions of~(\ref{sp4def}) take the form
\footnote{Care must be taken with the isometry algebra of $\Im(\KK)$,
corresponding to $\Im(\tr(\bphi))\ne0$.  Such elements can however also be
generated as commutators of elements of the form~(\ref{Ablock}), so we do not
consider them separately.}
\begin{equation}
A = 
\begin{pmatrix}
  \bphi-\frac12\,\rho\,\II& \AA\\
\noalign{\medskip}
  \BB& -\bphi^\dagger+\frac12\,\rho\,\II
\end{pmatrix}
\label{Ablock}
\end{equation}
where both $\AA$ and $\BB$ are Hermitian, $\tr(\bphi)=0$, and $\rho\in\RR$.
But generators of \hbox{$\so(k+2,2)$} take exactly the same form: $\bphi$
represents an element of $\so(k+1,1)$, $\AA$ and $\BB$ are (null)
translations, and $\rho$ is the dilation.  Direct computation shows that the
generators $A$ of \hbox{$\so(k+2,2)$} do indeed satisfy~(\ref{sp4def}); the
above construction therefore establishes the isomorphism
\begin{equation}
\so(k+2,2) \isom \sp(4,\KK)
\end{equation}
as claimed.

We can bring the representation~(\ref{Pdef}) into a more explicitly symplectic
form by treating $\XX$ as a vector-valued 1-form, and computing its Hodge
dual ${*}\XX$, defined by
\begin{equation}
{*}\XX = \XX\eps
\end{equation}
where
\begin{equation}
\eps = \begin{pmatrix} 0& 1\\ -1& 0 \end{pmatrix}
\end{equation}
is the Levi-Civita tensor in two dimensions.  Using the identity
\begin{equation}
\eps \XX \eps = \Xt{}^T
\end{equation}
we see that $P=P_0\,\II\otimes\eps$ takes the form
\begin{equation}
P = \begin{pmatrix} p\,\eps& {*}\XX\\ -({*}\XX)^T& q\,\eps \end{pmatrix}
\label{squarie}
\end{equation}
which is antisymmetric, and whose block structure is shown in
Figure~\ref{square}.  The diagonal blocks, labeled $00$ and $11$, are
antisymmetric, and correspond to $p$ and $q$, respectively, whereas the
off-diagonal blocks, labeled $01$ and $10$, contain equivalent information,
corresponding to ${*}\XX$.  Note that ${*}\XX$ does not use up all of the
degrees of freedom available in an off-diagonal block; the set of \textit{all}
antisymmetric $4\times4$ matrices is \textit{not} an irreducible
representation of $\sp(4,\KK)$.

The action of $\sp(4,\KK)$ on $P$ is given by
\begin{equation}
P \longmapsto AP + PA^T
\label{sp4act}
\end{equation}
for $A\in\sp(4,\KK)$, that is, for $A$ satisfying~(\ref{sp4def}).
\footnote{Thus,~(\ref{sp4act}) can be used if desired to determine the signs
in~(\ref{soact}).}
When working over $\KK=\RR$ or $\CC$, the action~(\ref{sp4act}) is just the
antisymmetric square
\begin{equation}
v \wedge w \longmapsto Av \wedge w + v \wedge Aw
\label{wedge2}
\end{equation}
of the natural representation $v\longmapsto Av$, with $v\in\KK^4$.

\begin{figure}
\centering
\hfill
\includegraphics[height=1in]{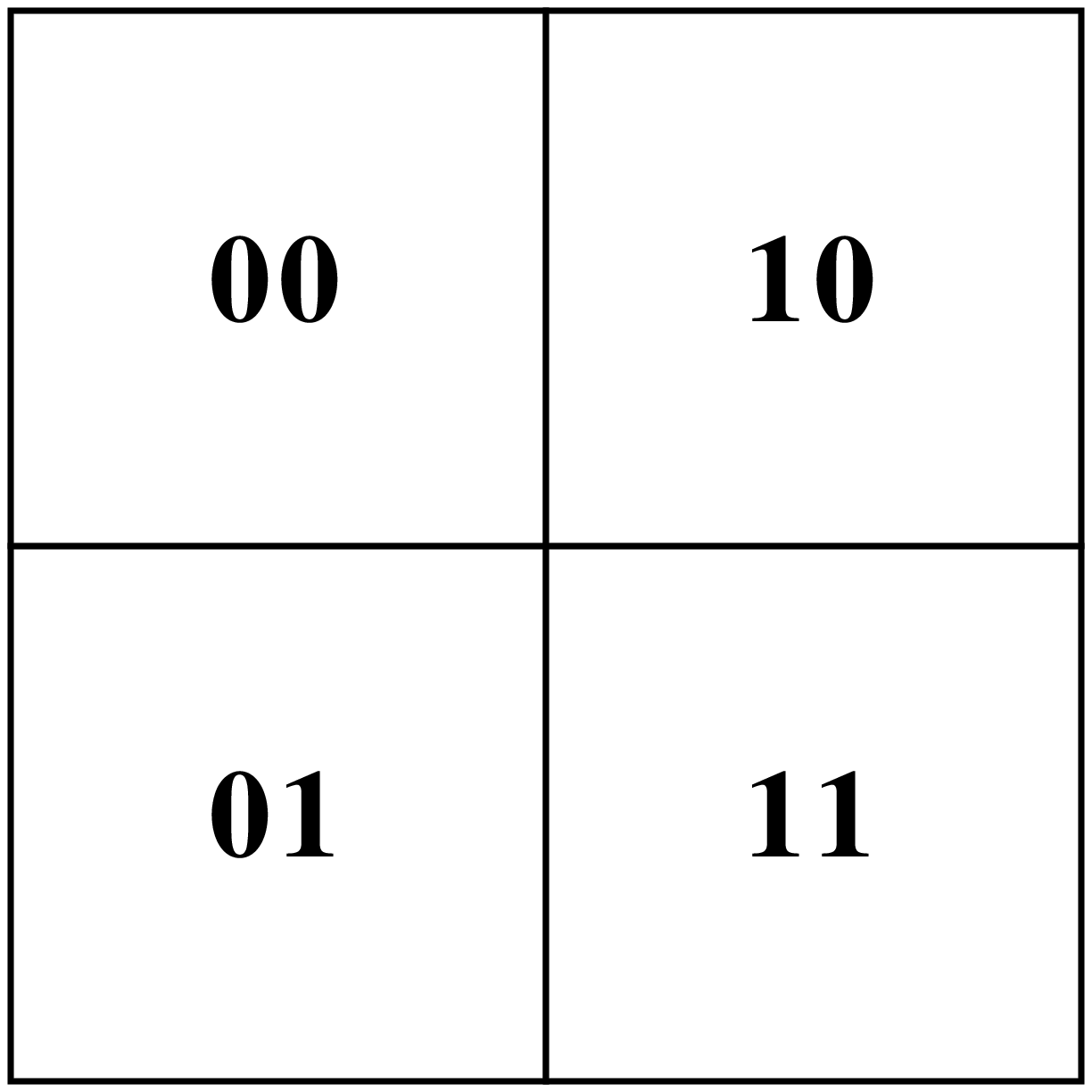}
\hfill
\includegraphics[height=1in]{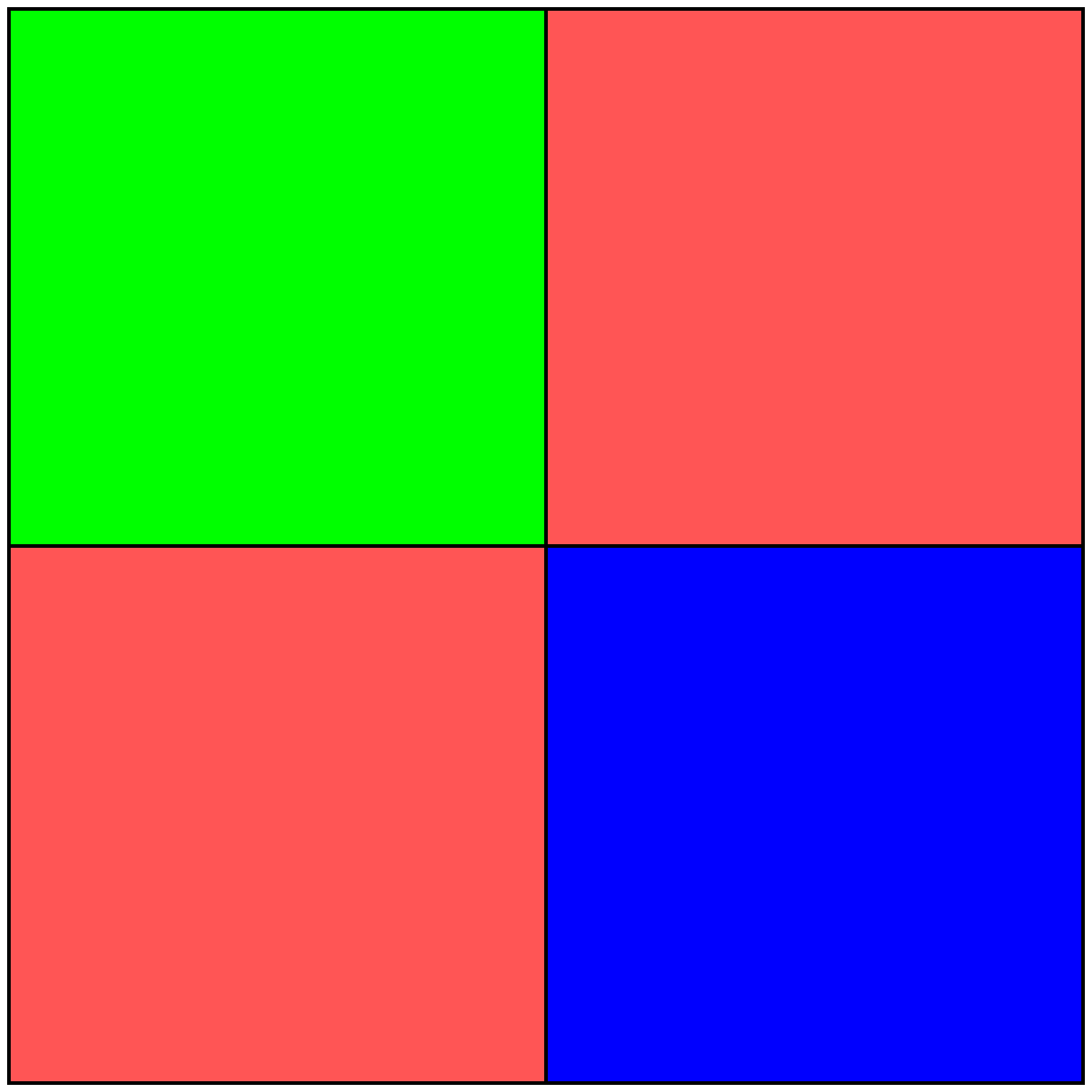}
\hfill\null
\caption{The block structure of a $4\times4$ antisymmetric matrix in terms of
$2\times2$ blocks.  A binary labeling of the blocks is shown on the left; on
the right, blocks with similar shading contain equivalent information.}
\label{square}
\end{figure}

\section{Cubies}
\label{cubies}

Before generalizing the above construction to the $3\times3$ magic square, we
first consider the analog of ${*}\XX$.  Let $\XXX\in\HHH$ be an element of the
Albert algebra, which we can regard as a vector-valued 1-form with components
$\XXX_a{}^b$, with $a,b\in\{1,2,3\}$.  The Hodge dual ${*}\XXX$ of $\XXX$ is a
vector-valued 2-form with components
\begin{equation}
(*\XXX)_{abc}=\XXX_a{}^m\epsilon_{mbc}
\label{Hodge}
\end{equation}
where $\epsilon_{abc}$ denotes the Levi-Civita tensor in three dimensions,
that is, the completely antisymmetric tensor satisfying
\begin{equation}
\epsilon_{123} =1
\end{equation}
and where repeated indices are summed over.  We refer to ${*}\XXX$ as a
\textit{cubie}.
We also introduce the dual of $\epsilon_{abc}$, the completely antisymmetric
tensor $\epsilon^{abc}$ satisfying
\begin{equation}
\epsilon_{mns}\epsilon^{mns} = 6
\label{eps6}
\end{equation}
and note the further identities
\begin{align}
\epsilon_{amn}\,\epsilon^{bmn} &= 2\,\delta_a{}^b \\
\epsilon_{abm}\,\epsilon^{cdm}
  &= \delta_a{}^c \,\delta_b{}^d - \delta_a{}^d \,\delta_b{}^c \\
\epsilon_{abc}\,\epsilon^{def}
  &= \delta_a{}^d\,\delta_b{}^e\,\delta_c{}^f
	+ \delta_b{}^d\,\delta_c{}^e\,\delta_a{}^f
	+ \delta_c{}^d\,\delta_a{}^e\,\delta_b{}^f \nonumber\\
  &\qquad - \delta_a{}^d\,\delta_c{}^e\,\delta_b{}^f
	- \delta_b{}^d\,\delta_a{}^e\,\delta_c{}^f
	- \delta_c{}^d\,\delta_a{}^e\,\delta_b{}^f
\label{epssq}
\end{align}
In particular, we have
\begin{equation}
({*}\XXX)_{amn}\epsilon^{bmn} = 2 \XXX_a{}^b
\label{sinv}
\end{equation}

Operations on the Albert algebra can be rewritten in terms of cubies.  For
instance,
\begin{align}
\tr\XXX &= \frac12\, \XXX_{abc} \,\epsilon^{abc} \\
\bigl({*}(\XXX\,\YYY)\bigr)_{abc}
  &= \frac12\, \XXX_{amn}\,\YYY_{pbc} \,\epsilon^{mnp} \\
\bigl({*}(\XXX\circ\YYY)\bigr)_{abc} &=
	\frac14 \bigl( \XXX_{amn}\,\YYY_{pbc} + \YYY_{amn}\,\XXX_{pbc} \bigr)
	\,\epsilon^{mnp} \\
\tr(\XXX\circ\YYY)
  &= \frac18 \bigl( \XXX_{amn}\,\YYY_{pbc} + \YYY_{amn}\,\XXX_{pbc} \bigr)
	\,\epsilon^{mnp}\,\epsilon^{bca} \nonumber\\
  &= \frac18 \bigl( \XXX_{amn}\,\YYY_{pbc} + \YYY_{pbc}\,\XXX_{amn} \bigr)
	\,\epsilon^{mnp}\,\epsilon^{bca} \\
(\tr\XXX)(\tr\YYY)
  &= \frac12\, \XXX_{abc}\,\YYY_{def} \,\epsilon^{abc}\,\epsilon^{def}
\end{align}
from which the components of ${*}(\XXX*\YYY)$ can also be worked out.  In the
special case where the components of $\XXX$ and $\YYY$ commute, contracting
both sides of~(\ref{epssq}) with $\XXX\otimes\YYY$ yields
\begin{equation}
\frac12 \XXX_c{}^m \YYY_d{}^n \,\epsilon_{amn} \,\epsilon^{bcd}
  = (\XXX*\YYY)_a{}^b
\label{cFreud}
\end{equation}
or equivalently
\begin{equation}
\bigl({*}(\XXX*\YYY)\bigr)_{abc}
  = \frac12 (\XXX_b{}^m \YYY_c{}^n - \XXX_c{}^m \YYY_b{}^n) \,\epsilon_{amn}
\end{equation}
providing two remarkably simple expressions for the Freudenthal product,
albeit only in a very special case.  We will return to this issue below.

\begin{lemma}
The action of $\phi\in\ee_6$ on cubies is given by
\begin{equation}
\XXX_a{}^m \epsilon_{mbc} \longmapsto
	\phi_a{}^m \XXX_m{}^n \epsilon_{nbc}
	+ \XXX_a{}^n \phi'_b{}^m \epsilon_{nmc}
	+ \XXX_a{}^n \phi'_c{}^m \epsilon_{nbm}
\label{e6cact}
\end{equation}
\label{e6lemma}
\end{lemma}

\begin{proof}
Consider the expression
\begin{equation}
Q_{nbc} =
\phi'_n{}^m \epsilon_{mbc} + \phi'_b{}^m \epsilon_{nmc} + \phi'_c{}^m \epsilon_{nbm}
\end{equation}
which is completely antisymmetric, and hence vanishes unless $n$, $b$, $c$ are
distinct.  But then
\begin{equation}
Q_{nbc} = \tr(\phi') \,\epsilon_{nbc}
\label{trphi}
\end{equation}
which vanishes, since $\tr(\phi')=-\tr(\phi)=0$.  Thus,~(\ref{e6act}) becomes
\begin{align}
\XXX_a{}^m \epsilon_{mbc}
  &\longmapsto \bigl(\phi_a{}^n \XXX_n{}^m
	+ \XXX_a{}^n \phi^\dagger_n{}^m \bigr) \,\epsilon_{mbc} \nonumber\\
  &= \phi_a{}^n \XXX_n{}^m \,\epsilon_{mbc}
	+ \XXX_a{}^n \phi'_b{}^m \epsilon_{nmc}
	+ \XXX_a{}^n \phi'_c{}^m \epsilon_{nbm}
\end{align}
as claimed, where we have used both~(\ref{phiadj}) and~(\ref{trphi}).
\end{proof}
A similar result holds for the action of $\phi'$.

\section{\boldmath The Symplectic Structure of $\ee_7$}
\label{sp6}

\begin{figure}
\centering
\hfill
\includegraphics[height=2in]{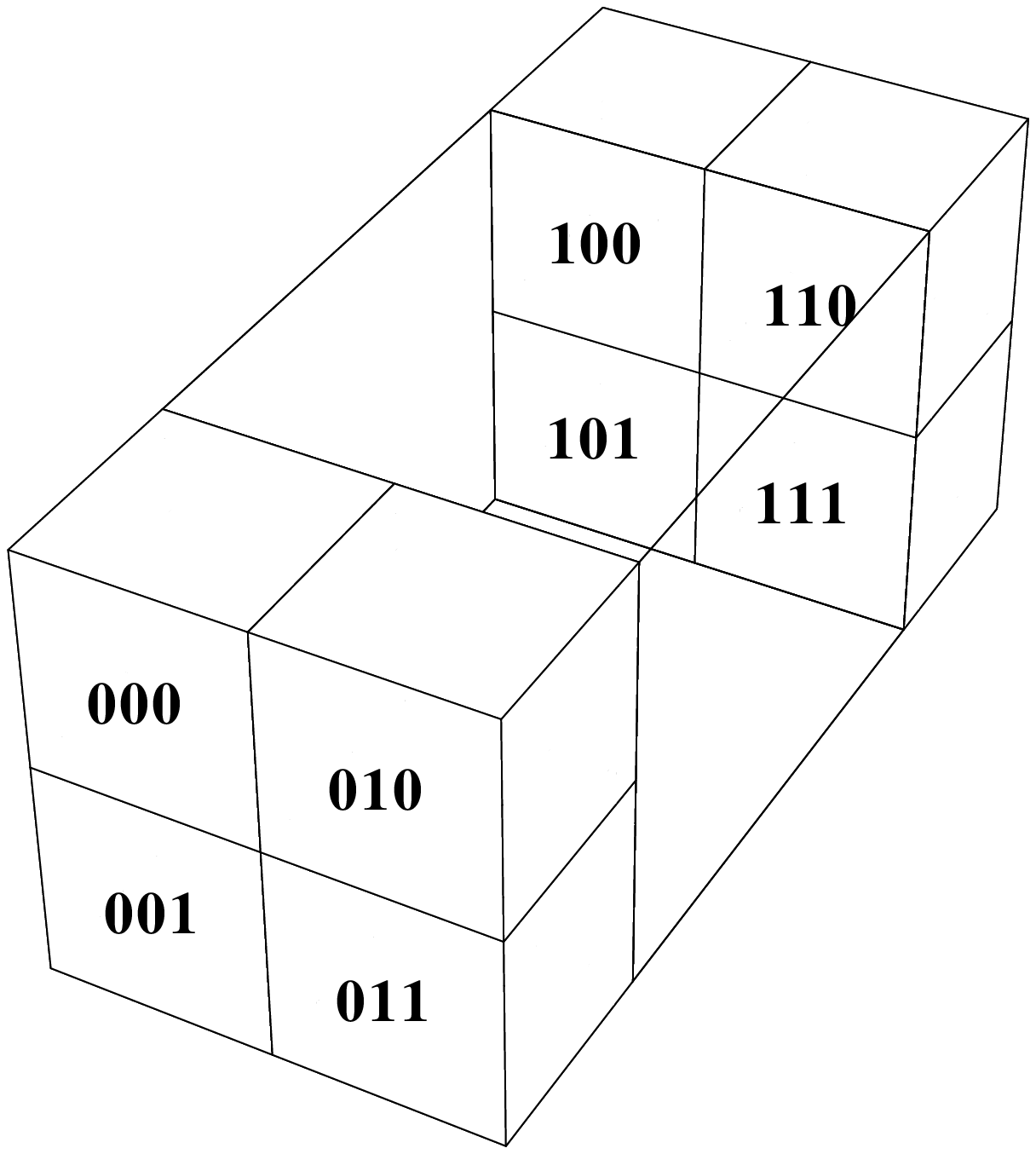}
\hfill
\includegraphics[height=2in]{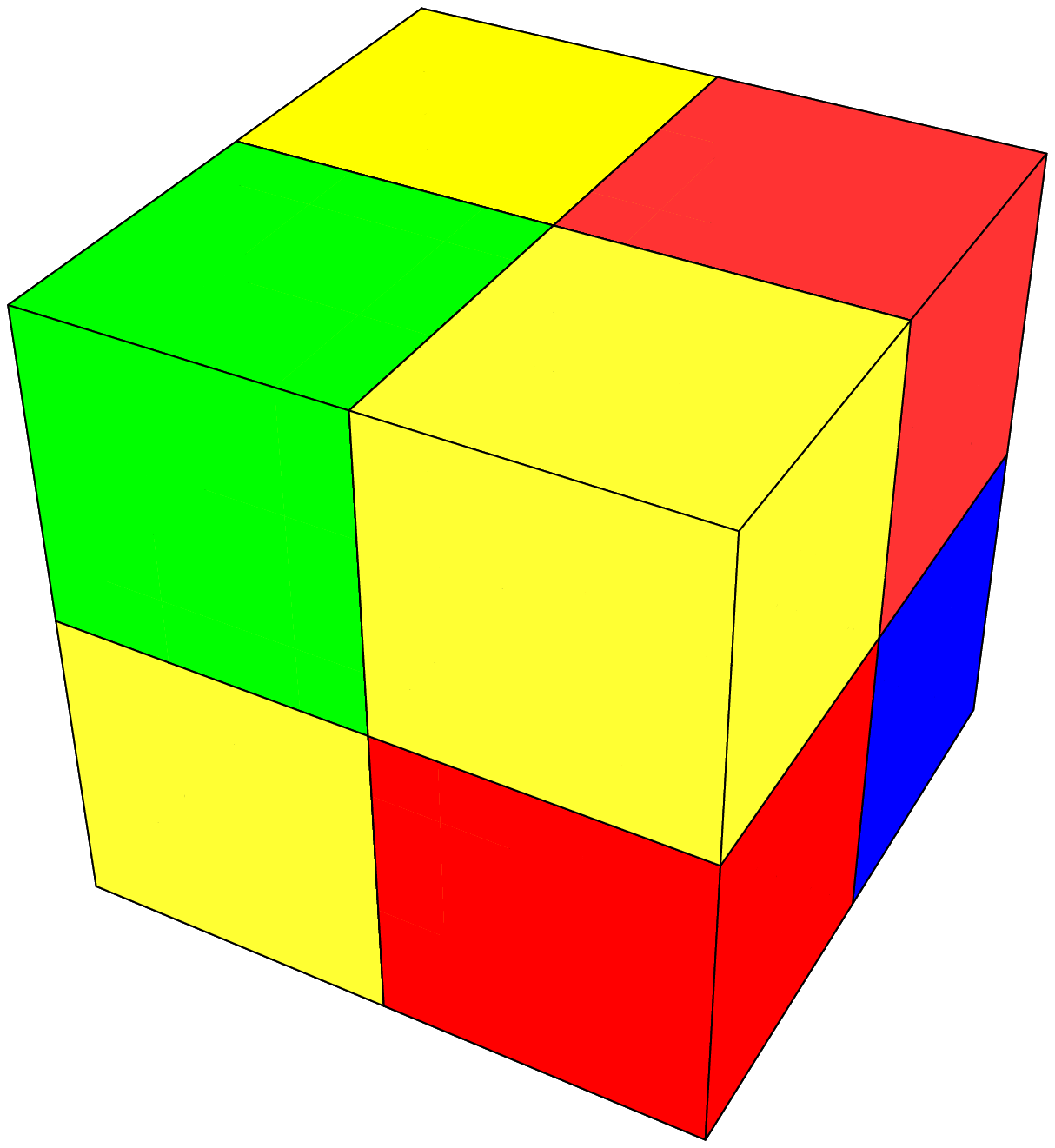}
\hfill\null
\caption{The block structure of a $6\times6\times6$ antisymmetric tensor in
terms of $3\times3\times3$ ``cubies''.  A binary labeling of the cubies is
shown on the pulled-apart cube on the left; on the right, cubies with similar
shading contain equivalent information.}
\label{cube}
\end{figure}

The representation~(\ref{ThetaDef}) can be written in block form, which we
also call~$\Theta$, namely
\footnote{The derivations $\gg_2\subset\ee_6$ require nested matrix
transformations of  the form~(\ref{e7mat}).}
\begin{equation}
\Theta =
\begin{pmatrix}
  \phi-\frac13\,\rho\,\III& \AAA\\
\noalign{\medskip}
  \BBB& \phi'+\frac13\,\rho\,\III
\label{e7mat}
\end{pmatrix}
\end{equation}
where $\III$ denotes the $3\times3$ identity matrix.  By analogy with
Section~\ref{sp4}, we would like $\Theta$ to act on ${*}\XXX$, which has 3
indices, and the correct symmetries to be an off-diagonal block of a rank 3
antisymmetric tensor $\Pc$, whose components make up a $6\times6\times6$ cube,
which we divide into $3\times3\times3$ cubies, as shown in Figure~\ref{cube};
compare Figure~\ref{square}.  We identity the diagonal cubies, labeled $000$
and $111$, with $p\,{*}\III$ and $q\,{*}\III$, respectively, the cubie labeled
$011$ with ${*}\XXX$, the cubie labeled $100$ with ${*}\YYY$, and then let
antisymmetry do the rest.  Explicitly, we have
\begin{equation}
\Pc_{abc} =
\begin{cases}
 p\, \epsilon_{abc} & a\le3,b\le3,c\le3 \\
 ({*}\YYY)_{\hat{a}bc} & a\ge4,b\le3,c\le3 \\
 ({*}\XXX)_{a\hat{b}\hat{c}} & a\le3,b\ge4,c\ge4 \\
 q\, \epsilon_{\hat{a}\hat{b}\hat{c}} & a\ge4,b\ge4,c\ge4
\end{cases}
\label{Pblocks}
\end{equation}
where we have introduced the convention that $\hat{a}=a-3$, and where the
remaining components are determined by antisymmetry.
\footnote{Note that $\Pc$ is a \textit{cube}, and has components $\Pc_{abc}$
with $a,b,c\in\{1,2,3,4,5,6\}$, whereas $\epsilon_{abc}$, ${*}\XXX_{abc}$, and
${*}\YYY_{abc}$ are the components of \textit{cubies}, which are subblocks of
$\Pc$, with $a,b,c\in\{1,2,3\}$.}

In the complex case, we could begin with the natural action of $\Theta$ on
6-component complex vectors, and then take the antisymmetric cube, that is, we
could consider the action
\begin{equation}
u\wedge v \wedge w \longmapsto
	\Theta u \wedge v \wedge w + u\wedge \Theta v \wedge w
	+ u\wedge v \wedge \Theta w
\label{wedge3}
\end{equation}
with $u,v,w\in\CC^6$, or equivalently
\begin{equation}
\Pc_{abc} \longmapsto
	\Theta_a{}^m \Pc_{mbc} + \Theta_b{}^m \Pc_{amc} + \Theta_c{}^m \Pc_{abm}
\label{Pact}
\end{equation}

\begin{lemma}
The action of the dilation $\Theta=(0,\rho,0,0)\in\ee_7$ on $\Pc$ is given
by~(\ref{Pact}).
\label{dilemma}
\end{lemma}

\begin{proof}
From~(\ref{e7mat}), we have
\begin{equation}
\Theta_a{}^b = \pm\frac13\,\rho\,\delta_a{}^b
\end{equation}
with the sign being negative for $a=b\le3$ and positive for $a=b\ge3$.
Thus,~(\ref{Pact}) becomes
\begin{equation}
\Pc_{abc} \longmapsto
	\pm\frac13\,\rho\,\Pc_{abc} \pm\frac13\,\rho\,\Pc_{abc}
	\pm\frac13\,\rho\,\Pc_{abc}
\end{equation}
where the signs depend on which of $a$, $b$, $c$ are ``small'' ($\le3$) or
``large'' ($\ge4$).  Examining~(\ref{Pblocks}), it is now easy to see that
$p\mapsto-\rho\,p$, $q\mapsto+\rho\,p$, $\XXX\mapsto+\frac\rho3\XXX$, and
$\YYY\mapsto-\frac\rho3\YYY$, exactly as required
by~(\ref{FreudX})--(\ref{qFreud}).
\end{proof}

\begin{lemma}
If the elements of $\AAA,\BBB\in\HHH$ commute with those of $\Pc$, then the
action of the translations $\Theta=(0,0,\AAA,0)$ and $\Theta=(0,0,0,\BBB)$ on
$\Pc$ is given by~(\ref{Pact}).
\label{trlemma}
\end{lemma}

\begin{proof}
Set $\Theta=(0,0,\AAA,0)$ and consider the action of $\Theta$ on $p$, $\XXX$,
$\YYY$, and $q$, needing to verify~(\ref{FreudX})--(\ref{qFreud}) with
$\phi=0$, $\rho=0$, and $\BBB=0$.  From~(\ref{e7mat}), we have
\begin{equation}
\Theta_a{}^b =
\begin{cases}
\AAA_a{}^{\hat{b}} & a\le3, b\ge4 \\
0 & \hbox{otherwise}
\end{cases}
\end{equation}
Since $\AAA$ has one ``small'' index and one ``large'' index, it acts as a
lowering operator, e.g.\ mapping cubie $100$ to $000$, and thus maps
$q\mapsto\XXX\mapsto\YYY\mapsto p$.  In particular, this confirms the lack of
a term involving $\AAA$ in~(\ref{qFreud}).
Considering terms involving $q$, we look at cubie $011$, where the only
nonzero term of~(\ref{Pact}) is
\begin{equation}
({*}\XXX)_{abc} \longmapsto \AAA_a{}^m q\,\epsilon_{mbc} = q\,({*}\AAA)_{abc}
\end{equation}
which verifies~(\ref{FreudX}) in this case.

We next look at cubie $000$, where~(\ref{Pact}) becomes
\begin{equation}
p\,\epsilon_{abc} \longmapsto 
	\AAA_a{}^m ({*}\YYY)_{mbc} + \AAA_b{}^m ({*}\YYY)_{mca}
	+ \AAA_c{}^m ({*}\YYY)_{mab}
\end{equation}
which is clearly antisymmetric, so we can use~(\ref{eps6}) and~(\ref{sinv}) to
obtain
\begin{equation}
p \longmapsto \frac12\, \AAA_a{}^m ({*}\YYY)_{mbc} \,\epsilon^{abc}
  = \AAA_a{}^m \YYY_m{}^a = \tr(\AAA\YYY) = \tr(\AAA\circ\YYY)
\end{equation}
which is~(\ref{pFreud}), where we have used commutativity only in the last
equality.

Finally, turning to cubie $100$,~(\ref{Pact}) becomes
\begin{align}
({*}\YYY)_{abc}
  &\longmapsto \AAA_b{}^m ({*}\XXX)_{cam} + \AAA_c{}^m ({*}\XXX)_{bma} \nonumber\\
  &\qquad= \AAA_b{}^m \XXX_c{}^n \,\epsilon_{nam}
	+ \AAA_c{}^m \XXX_b{}^n \,\epsilon_{nma}
\end{align}
or equivalently, using~(\ref{sinv}) and~(\ref{cFreud}),
\begin{equation}
2\, \YYY_a{}^b
  \longmapsto 2\, \AAA_e{}^m \XXX_f{}^n \,\epsilon_{amn} \,\epsilon^{bef}
  = 4\, (\XXX*\YYY)_a{}^b
\end{equation}
which is~(\ref{FreudY}).

This entire argument can be repeated with only minor changes if
$\Theta=(0,0,0,\BBB)$.
\end{proof}

Over $\RR$ or $\CC$, we're done; Lemmas~\ref{e6lemma},~\ref{dilemma},
and~\ref{trlemma} together suffice to show that the action~(\ref{Pact}) is the
same as the Freudenthal action~(\ref{FreudX})--(\ref{qFreud}).  Unfortunately,
the action~(\ref{Pact}) fails to satisfy the Jacobi identity over $\HH$ or
$\OO$.  However, we can still use Lemmas~\ref{e6lemma},~\ref{dilemma},
and~\ref{trlemma} to reproduce the Freudenthal action in those cases, as
follows.

\begin{lemma}
The action of $\Theta=(\phi,0,0,0)\in\ee_7$ on $\Pc$ is determined by
\begin{equation}
\Pc_{abc} \longmapsto
	\Theta_a{}^m \Pc_{mbc} + \Pc_{amc} \Theta_b{}^m + \Pc_{abm} \Theta_c{}^m
\label{cubeact}
\end{equation}
when acting on elements of the form~(\ref{Pblocks}), which extends to all
of $\ee_7$ by antisymmetry.
\label{e6lemma2}
\end{lemma}

\begin{proof}
From~(\ref{e7mat}), we have
\begin{equation}
\Theta_a{}^b =
\begin{cases}
\phi_a{}^b & a\le3, b\le3 \\
\phi'_{\hat{a}}{}^{\hat{b}} & a\ge4, b\ge4 \\
0 & \hbox{otherwise}
\end{cases}
\label{Thetaphi}
\end{equation}
Inserting~(\ref{Thetaphi}) into~(\ref{cubeact}) now yields
precisely~(\ref{e6cact}) when acting on $\XXX$; the argument for the action on
$\YYY$ is similar.  Furthermore, using a argument similar to that used to
prove Lemma~\ref{e6lemma} to begin with,~(\ref{Pact}) acts on $p$ via
\begin{equation}
p\,\epsilon_{abc} \longmapsto
	\phi_a{}^m p\,\epsilon_{mbc} + \phi_b{}^m p\,\epsilon_{amc}
	+ \phi_c{}^m p\,\epsilon_{abm}
\end{equation}
which is completely antisymmetric in $a$, $b$, $c$, and therefore proportional
to $\tr(\phi)=0$.  The argument for the action on $q$ is similar, with $\phi$
replaced by $\phi'$.  Although~(\ref{cubeact}) itself is only antisymmetric in
its last two indices, that suffices to define an action on cubies $000$, $011$,
$100$, and $111$; the action on the remaining 4 cubies is uniquely determined
by requiring that antisymmetry be preserved.
\end{proof}

We now have all the pieces, and can state our main result.

\begin{theorem}
The Lie algebra $\ee_7$ acts symplectically on cubes, that is,
$\ee_6\subset\ee_7$ acts on cubes via~(\ref{cubeact}), as do real translations
and the dilation, and all other $\ee_7$ transformations can then be constructed
from these transformations using linear combinations and commutators.
\label{Thm}
\end{theorem}

\begin{proof}
Lemmas~\ref{dilemma} and~\ref{trlemma} are unchanged by the use
of~(\ref{cubeact}) rather than~(\ref{Pact}), since the components of $\Theta$
commute with those of $\Pc$ in both cases, and Lemma~\ref{e6lemma2} verifies
that $\ee_6$ acts via~(\ref{cubeact}), as claimed.  It only remains to show
that the remaining generators of $\ee_7$ can be obtained from these elements
via commutators.

Using~(\ref{FreudX})--(\ref{qFreud}), it is straightforward to compute the
commutator of two $\ee_7$ transformations of the form~(\ref{ThetaDef}).
Letting $\phi=\QQQ\in\ee_6$ be a boost, so that $\QQQ^\dagger=\QQQ$ and
$\tr(\QQQ)=0$, and using the identity
\begin{equation}
-(\AAA\circ\BBB)*\XXX
  = \bigl(\BBB-\tr(\BBB)\III\bigr)\circ(\AAA*\XXX) + \AAA*(\BBB\circ\XXX)
\end{equation}
for any $\AAA,\BBB,\XXX\in\HHH$, we obtain
\begin{equation}
\bigl[(0,0,\AAA,0),(\QQQ,0,0,0)\bigr] = (0,0,\AAA\circ\QQQ,0)
\end{equation}
We can therefore obtain the null translation $(0,0,\QQQ,0)$ for any
\textit{tracefree} Albert algebra element $\QQQ$ as the commutator of
$(0,0,\III,0)$ and $(\QQQ,0,0,0)$; a similar argument can be used to construct
the null translation $(0,0,0,\QQQ)$.
\end{proof}

Thus, \textit{all} generators of $\ee_7$ can be implemented either as a
symplectic transformation on cubes via~(\ref{cubeact}), or as the commutator
of two such transformations.

\section{Discussion}
\label{discuss}

We have showed that the algebraic description of the minimal representation of
$e_7$ introduced by Freudenthal naturally corresponds geometrically to a
symplectic structure.  Along the way, we have emphasized both the similarities
and differences between $\ee_7$ and $\so(10,2)$.  Both of these algebras are
\textit{conformal}; their elements divide naturally into generalized rotations
($\ee_6$ or $\so(9,1)$, respectively), translations, and a dilation.  Both act
naturally on a representation built out of vectors ($3\times3$ or $2\times2$
Hermitian octonionic matrices, respectively), together with two additional
real degrees of freedom ($p$ and $q$).  In the $2\times2$ case, the
representation~(\ref{Pdef}) contains just one vector; in the $3\times3$
case~(\ref{Pcdef}), there are two.  This at first puzzling difference is fully
explained by expressing both representations as antisymmetric tensors, as
in~(\ref{squarie}) and~(\ref{Pblocks}), respectively, and as shown
geometrically in Figures~\ref{square} and~\ref{cube}.

In the complex case, we have shown that the symplectic action~(\ref{Pact})
exactly reproduces the Freudenthal action~(\ref{FreudX})--(\ref{qFreud}).  The
analogy goes even further.  In $2n$ dimensions, there is a natural map taking
two $n$-forms to a $2n\times 2n$ matrix.  When acting on $\Pc$, this map takes
the form
\begin{equation}
\Pc \longmapsto \Pc_{acd}\Pc_{efb}\,\epsilon^{acdefb}
\label{PstarP}
\end{equation}
where $\epsilon$ now denotes the volume element in six dimensions, that is,
the completely antisymmetric tensor with $\epsilon^{123456}=1$.  It is not
hard to verify that, in the complex case,~(\ref{PstarP}) is (a multiple of)
$\Pc*\Pc$, as given by~(\ref{PstarP0})--(\ref{PstarP4}).  Similarly, the
quartic invariant~(\ref{quartic}) can be expressed in the complex case as
\begin{equation}
J \sim \Pc_{gab}\Pc_{cde}\Pc_{fhi}\Pc_{jkl}
	\,\epsilon^{abcdef}\,\epsilon^{ghijkl}
\label{quarticC}
\end{equation}
up to an overall factor.

Neither the form of the action~(\ref{Pact}), nor the
expressions~(\ref{PstarP}) and~(\ref{quarticC}), hold over $\HH$ or $\OO$.
This failure should not be a surprise, as trilinear tensor products are not
well defined over $\HH$, let alone $\OO$.  Nonetheless, Theorem~\ref{Thm} does
tell us how to extend~(\ref{Pact}) to the octonions.  Although it is also
possible to write down versions of~(\ref{PstarP}) and~(\ref{quarticC}) that
hold over the octonions, by using case-dependent algorithms to determine the
order of multiplication, it is not clear that such expressions have any
advantage over the original expressions~(\ref{PstarP0})--(\ref{PstarP4})
and~(\ref{quartic}) given by Freudenthal.

Despite these drawbacks, it is clear from our construction that $\ee_7$ should
be regarded as a natural generalization of the traditional notion of a
symplectic Lie algebra, and fully deserves the name $\sp(6,\OO).$

\section*{Acknowledgments}

We thank John Huerta for discussions, and for coining the term ``cubies''.
This work was supported in part by the John Templeton Foundation.

\goodbreak


\end{document}